\documentclass{amsart}
\usepackage{amsfonts}
\usepackage{latexsym}
\usepackage{hyperref}
\usepackage{amssymb}
\usepackage{amsmath}
\usepackage{enumerate}
\usepackage{latexsym,color,amsmath,amsthm,amssymb,amscd,amsfonts}
\usepackage{graphicx}
\usepackage[normalem]{ulem}

\usepackage[utf8]{inputenc}
\usepackage[T1]{fontenc}

\newcommand{\vol}[1]{\mathrm{vol}\left(#1\right)}

\newtheorem{thm}{Theorem}[section]
\newtheorem{cor}[thm]{Corollary}
\newtheorem{lemma}[thm]{Lemma}

\newtheorem{propo}[thm]{Proposition}
\theoremstyle{remark}
\newtheorem*{rmk}{Remark}

\begin{document}

%%%%%%%%%%%%%%%%%%%%%%%%%%%%%%%%%%%%%%%%%%%%%5

\title[Isoperimetric local Bollob\'as-Thomason inequalities]{On isoperimetric local-Bollob\'as-Thomason inequalities}

\author[L. J. Al\'ias]{Luis J. Al\'ias}
\email{ljalias@um.es}
\address{Departamento de Matem\'aticas, Universidad de Murcia, E-30100 Espinardo, Murcia, Spain}

\author[B. Gonz\'alez Merino]{Bernardo Gonz\'alez Merino}
\email{bgmerino@um.es}
\address{Departamento de Ingeniería y Tecnología de Computadores, Área de Matemática Aplicada,
Universidad de Murcia, E-30100 Espinardo, Murcia, 
Spain}

\author[B. Mar\'in Gimeno]{Beatriz Mar\'in Gimeno}
\email{b.maringimeno@um.es}
\address{Departamento de Matem\'aticas, Universidad de Murcia, E-30100 Espinardo, Murcia, Spain}


\thanks{This research has been partially supported by the grants PID2021-124157NB-I00, PID2022-136320NB-I00, PID2025-167696NB-I00, and PID2025-169631NB-I00, all of them funded by MICIN/AEI, Spain. The third author is also supported by Contratos
Predoctorales FPU-Universidad de Murcia 2024, Spain}

\date{\today}

\begin{abstract}
We prove the following isoperimetric-type inequality: for every convex body $K$ in $\mathbb R^n$ and some $\sigma\subset[n]:=\{1,\dots,n\}$ there exists a suitable Hanner polytope $B_K$ with the same volume as $K$ and such that the volume of each of its orthogonal projections onto every subspace whose basis is formed by the canonical vectors $\{e_i:i\in\tau\cup([n]\setminus\sigma)\}$, for every $\tau\subseteq\sigma$, bounds from below the volume of the corresponding projections of $K$.
\end{abstract}

\maketitle

\section{Introduction}

The classical isoperimetric inequality \cite{O78} states that for any $n$-dimensional convex body $K$ (i.e. a convex compact set of $\mathbb{R}^n$ with non-empty interior) there exists an Euclidean ball $B$ such that 
\begin{equation}\label{eq:  classical isoperimetric inequality}
    \mathrm{vol}_n(K)=\mathrm{vol}_n(B)\quad\text{with}\quad S(K) \geq S(B). 
\end{equation}
Above, $\mathrm{vol}_n(\cdot)$ denotes the $n$-dimensional volume or Lebesgue measure, $S(\cdot)$ denotes the surface area 
%of a $n$-dimensional convex body 
and $B$ is a rescaling $rB^n_2$ of the Euclidean unit ball $B_2^n:=\{x\in\mathbb{R}^n:\Vert x\Vert_2\leq1\}$ and where $r^n:=\vol K/\vol {B^n_2}$. We will denote by $\mathcal{K}^n$ the set of all $n$-dimensional convex bodies in $\mathbb{R}^n$. From now on, if we write $\vol{\cdot}$ instead of $\mathrm{vol}_n(\cdot)$ we understand that the volume is computed with respect to the dimension of the evaluated set.

Inequality \eqref{eq:  classical isoperimetric inequality} motivated several results in modern Convex Geometric Analysis. For instance, Petty projection inequality \cite{P71} relates the volumes of $K$ and its polar projection body, and it gives an affinely invariant strengthening of the isoperimetric inequality. Other isoperimetric-type inequalities can be found, for instance, in \cite{A79,B10,G20,H09,1HLPRY25,2HLPRY25,HL25,L86,LYZ00,MY78,  PP12, X07, Z91}. 

Bollob\'as and Thomason \cite{BT95} proved a strong isoperimetric result by comparing projections of a convex body $K$ to those of a suitable box depending on $K$. They proved that for any $K\in\mathcal K^n$ there exists a coordinate box $C_K$ such that
\begin{equation}\label{eq:BolloThoma}
    \vol{K}=\vol{C_K}\quad\text{and}\quad \vol{P_{H_{\tau}}K}\geq\vol{P_{H_\tau}C_K}
\end{equation}
for any $\tau\subseteq[n]:=\{1,\ldots,n\}$. Above, $P_{H_{\tau}}K$ denotes the orthogonal projection of $K$ onto the subspace $H_\tau:=\langle \{e_i:i\in\tau\}\rangle$, where $\{e_i\}_{i=1}^n$ is the canonical basis of $\mathbb{R}^n$. Liakopoulos \cite{L19} proved the dual counterpart to this result for sections and for a suitable $n$-dimensional crosspolytope as extremizer. 

Our main result is an isoperimetric-type result that compares projections of a convex body $K$ and a suitable 
% Hanner polytope
convex body $B_K$ onto every coordinate subspace containing a fixed subspace. 
\begin{thm}\label{thm: isoperimetric main projections}
     Let $K\in\mathcal{K}^n$ and $\sigma\subset [n]$. Then, there exist some $c_i>0$, $i\in\sigma$, and a convex body $B_K$ of the form
    $$B_K=\mathrm{conv}\left(\left\{\sum_{i\in\sigma}c_i[-e_i,e_i],L
    \right\}\right)
    $$
    such that
    \begin{itemize}
        \item[i)] $\vol{K}=\vol{B_K}$,
        \item[ii)] $\vol{P_{H_\sigma^\bot\oplus H_\tau }K}\geq\vol{P_{H_\sigma^\bot\oplus H_\tau }B_K}$ for every $\tau\subseteq \sigma$ and
        \item[iii)] $\vol{L}=\vol{P_{H_\sigma^\bot}K}$,
    \end{itemize}
    where $L$ is a $(n-|\sigma|)$-dimensional convex body contained in $H_\sigma^\bot,$ with $0\in L$.
\end{thm}
Let us remember that for any given $X\subset\mathbb R^n$, we denote by $\mathrm{conv}(X)$ the convex hull of $X$. For any $x,y\in\mathbb R^n$, let $[x,y]:=\mathrm{conv}(\{x,y\})$ be the line segment with extreme points $x,y$. For any $K,C\in\mathcal K^n$ the Minkowski sum of $K$ and $C$ is given by $K+C:=\{x+y:x\in K,y\in C\}$. Moreover, $t\cdot K:=\{tx:x\in K\}$ for every $t\geq 0$.

Notice that $B_K$ can be chosen to be a suitable Hanner polytope (see Corollary \ref{cor: Hanner}). Hanner polytopes interest partly relies on their connection to Mahler's conjecture. Indeed, they are supposed to be the only extremizers in the central version of it \cite{IS17}. 

Bollob\'as and Thomason's proof of \eqref{eq:BolloThoma} depends on the following inequality
\begin{equation}\label{eq: Bollobas-Thomason}
    \vol{K}^s\leq\prod_{i=1}^m\vol{P_{H_{\sigma_i}}K},
\end{equation}
see \cite[Thm.~2]{BT95}. Above, $K\in\mathcal K^n$ and $(\sigma_1,\dots,\sigma_m)$ is a $s$-cover of $[n]$. We say that $(\sigma_1,\dots,\sigma_m)$ is a $s$-cover of $\sigma\subseteq[n]$, if $\sigma_i\subseteq \sigma$, $i=1,\dots,m$, and $|\{i:j\in\sigma_i\}|=s$ for every $j\in\sigma$. Notice that \eqref{eq: Bollobas-Thomason} recovers the classical Loomis-Whitney inequality \cite{LM49} when $\sigma_i=[n]\setminus\{i\}$ for every $i\in[n]$.

In order to explain the characterization of the equality cases of \eqref{eq: Bollobas-Thomason}, we need further definitions. Given $(\sigma_1,\dots,\sigma_m)$ a $s$-cover of $\sigma\subseteq [n]$, we say that $(\overline{\sigma}_1,\dots,\overline{\sigma}_\ell)$ is the $1$-cover of $\sigma$ induced by the former $s$-cover, if $(\overline{\sigma}_1,\dots,\overline{\sigma}_\ell)$ are all the different non-empty possible subsets of the form $\cap_{j=1}^m \sigma_j^{\varepsilon(j)}$, where $\varepsilon(j)\in\{0,1\}$, $\sigma_j^0=\sigma_j$ and $\sigma_j^1=\sigma\setminus\sigma_j$ (see \cite{BT95, BKX23}). The equality case of \eqref{eq: Bollobas-Thomason} is characterized by the convex bodies of the form $K=\sum_{i=1}^\ell P_{H_{\overline{\sigma}_i}}K$, where $(\overline{\sigma}_1,\ldots,\overline{\sigma}_\ell)$ is the $1$-cover of $[n]$ induced by the $s$-cover of $[n]$, $(\sigma_1,\ldots,\sigma_m)$.

Years later, Brazitikos, Giannopoulos, and Liakopoulos \cite{BGL18} studied a generalization of these type of inequalities by taking $(\sigma_1,\ldots,\sigma_m)$ a $s$-cover of a proper subset $\sigma$ of $[n]$ and they called them local inequalities. Alonso-Guti\'errez, Bernu\'es, Brazitikos, and Carbery \cite{ABBC21} as well as Manui, Ndiaye, and Zvavitch \cite[Thm.~4]{MNZ24} have recently proven the following local Bollob\'as-Thomason type inequality
\begin{equation}\label{eq:ineqProj}
\vol{K}^{m-s}\vol{P_{H_\sigma^\bot}K}^s \leq \frac{\prod_{i=1}^m{n-|\sigma_i| \choose n-|\sigma|}}{{n\choose n-|\sigma|}^{m-s}} \prod_{i=1}^m\vol{P_{H_{\sigma_i}^\bot}K}.
\end{equation}
The previous inequality holds for every $K\in\mathcal K^n$, $\sigma\subset[n]$, and $(\sigma_1,\dots,\sigma_m)$ a $s$-cover of $\sigma$. Inequalities \eqref{eq:ineqProj} play a fundamental role in the proof of Theorem \ref{thm: isoperimetric main projections}. They attain equality when $K$ is a suitable Hanner polytope (see \cite[Rem.~5]{MNZ24}). However, it remains open to characterize the equality cases. For the sake of completeness, we prove that characterization below.

\begin{thm}\label{thm:LocalBollobasThomasonineq}
    Let $K\in\mathcal K^n$, $\sigma\subset[n]$, $(\sigma_1,\dots,\sigma_m)$ a $s$-cover of $\sigma$, and let $(\overline{\sigma}_1,\dots,\overline{\sigma}_k)$ be the cover induced by $(\sigma\setminus\sigma_1,\dots,\sigma\setminus\sigma_m)$. Then equality holds in \eqref{eq:ineqProj} if and only if there exists $y_0\in\mathbb{R}^n$, for every $y\in P_{H_\sigma^\bot}K$
    \begin{equation}\label{eq:Charact(1)}
        K\cap\left(y+H_\sigma\right) = z_1(y)+\sum_{j=1}^k P_{H_{\overline{\sigma}_j}\oplus H_\sigma^\bot}(K)\cap\left(y+H_{\overline{\sigma}_j}\right)
    \end{equation}
    for some $z_1(y)\in\mathbb R^n$,
    \begin{equation}\label{eq:Charact(2)}
    \begin{split}
        \sum_{j\in M_i} & P_{H_{\overline{\sigma}_j} \oplus H_\sigma^\bot}(K)\cap\left(y+H_{\overline{\sigma}_j}\right) \\
        &= z_2(y,i) + \left(1-\|y-y_0\|_{P_{H_\sigma^\bot}K-y_0}\right)\left( P_{H_{\sigma\setminus\sigma_i}\oplus H_\sigma^\bot}(K)\cap\left(y_0+H_{\sigma\setminus\sigma_i}\right)\right)
        % & = z_2(y,i)+(1-\|y-y_0\|_{P_{H_\sigma^\bot}K})\sum_{j\in M_i}P_{H_{\overline{\sigma}_j}\oplus H_\sigma^\bot}(K)\cap \left(y_0+H_{\overline{\sigma}_j}\right)
    \end{split}
    \end{equation}
    for every $i=1,\dots,m$ and some $z_2(y,i)\in\mathbb R^n$, and where $M_i:=\{j\in[k]:\overline{\sigma}_j\subset \sigma\setminus\sigma_i\}$, and
    \begin{equation}\label{eq:Charact(3)}
    \begin{split}
        &\vol{\sum_{j\in M_i}  P_{H_{\overline{\sigma}_j}\oplus H_\sigma^\bot}(K)\cap\left(y+H_{\overline{\sigma}_j}\right)}^\frac{1}{|\sigma|-|\sigma_i|} \\
        & = L_i 
        % \left(1-\|y-y_0\|_{P_{H_\sigma^\bot}K}\right)
        \vol{\sum_{j\in M_1}P_{H_{\overline{\sigma}_j}\oplus H_\sigma^\bot}(K)\cap\left(y_0+H_{\overline{\sigma}_j}\right)}^\frac{1}{|\sigma|-|\sigma_1|}
    \end{split}
    \end{equation}
for every $i\in[m]$ and for some absolute constants $L_i>0$.
\end{thm}

The paper is organized as follows. In Section \ref{sec:MainResult} we prove Theorem \ref{thm: isoperimetric main projections}. Afterwards in Section \ref{sec: characterization equality case} we prove the characterization of the equality case \eqref{eq:ineqProj}, i.e.~Theorem \ref{thm:LocalBollobasThomasonineq}. Finally in Appendix \ref{Sec:Appendix} we include for the sake of clearness the proof of \eqref{eq:ineqProj} as proven in \cite[Thm.~4]{MNZ24}.

\section{Isoperimetric local Bollob\'as-Thomason inequalities}\label{sec:MainResult}

Before proving Theorem \ref{thm: isoperimetric main projections}, we explain some facts about covers. First of all, we say that $(\sigma_1,\dots,\sigma_m)$ is a cover of $\sigma$ if it is a $s$-cover of $\sigma$ for some $s\geq 1$. Second,
we say that any $s$-cover $(\sigma_1,\dots,\sigma_m)$  of $\sigma\subseteq[n]$ is irreducible if $(\sigma_1,\dots,\sigma_m)$ cannot be separated into two subsets forming each of them a cover of $\sigma$. Otherwise, we call them reducible. The importance of irreducible covers relies in two facts: there is a finite number of irreducible covers for a given $\sigma\subseteq[n]$ (see \cite{BT95,G73}), and once \eqref{eq:ineqProj} is proven for every such irreducible cover, then one can immediately derive \eqref{eq:ineqProj} by just decomposing the corresponding cover into irreducible covers.

\begin{proof}[Proof of Theorem \ref{thm: isoperimetric main projections}]
Let $\sigma\subset[n]$ and $(\sigma_1,\ldots,\sigma_m)$ be any $s$-cover of $\sigma$. Then, by \eqref{eq:ineqProj} it follows that
\begin{equation}\label{eq: ineq 1 localBT}
   \vol{K}^{m-s}\vol{P_{H_\sigma^\bot}K}^s \leq \frac{\prod_{j=1}^m{n-|\sigma_j|\choose n-|\sigma|}}{{n\choose n-|\sigma|}^{m-s}}\prod_{j=1}^m\vol{P_{H_{\sigma_j}^\bot}K}.
\end{equation}
Dividing both sides by $\mathrm{vol}(P_{H_\sigma^\bot}K)^m$, which is not equal to zero since $K$ has dimension $n$, and taking into account that $H_{\sigma_j}^\bot=H_\sigma^\bot\oplus H_{\sigma\setminus\sigma_j}$
for every $j=1,\ldots,m$, inequality \eqref{eq: ineq 1 localBT} can be rewritten as
\begin{equation}\label{eq: ineq 2 localBT}
   \left({n\choose n-|\sigma|}\frac{\vol{K}}{\vol{P_{H_\sigma^\bot}K}}\right)^{m-s} \leq \prod_{j=1}^m{n-|\sigma_j|\choose n-|\sigma|}\frac{\vol{P_{H_{\sigma}^\bot\oplus H_{\sigma\setminus\sigma_j}}K}}{\vol{P_{H_\sigma^\bot}K}}. 
\end{equation}

Now let $\tau\subseteq\sigma$ and $(\tau_1,\ldots,\tau_m)$ be any $s$-cover of $\tau$. Applying \eqref{eq:ineqProj} to the convex body $\widehat{K}:=P_{H_\sigma^\bot\oplus H_\tau}K$ of dimension $n-|\sigma|+|\tau|$ and to the $s$-cover of $\tau$, $(\tau_1,\ldots,\tau_m)$ 
we have that
\begin{equation}\label{eq: local para el cuerpo proyectado}
    \vol{\widehat{K}}^{m-s}\vol{P_{H_\tau^{\widehat{\bot}}}\widehat{K}}^s \leq \frac{\prod_{j=1}^m{n-|\sigma|+|\tau|-|\tau_j| \choose n-|\sigma|}}{{n-|\sigma|+|\tau|\choose n-|\sigma|}^{m-s}} \prod_{i=1}^m\vol{P_{H_{\tau_j}^{\widehat{\bot}}}\widehat{K}},
\end{equation}
where $H_\tau^{\widehat{\bot}}$ and $H_{\tau_j}^{\widehat{\bot}}$  denote the orthogonal subspaces of $H_\tau$ and $H_{\tau_j}$, respectively, within the subspace $H_\sigma^\bot\oplus H_\tau$, and $P_{H_\tau^{\widehat{\bot}}}$ and $P_{H_{\tau_j}^{\widehat{\bot}}}$  the corresponding orthogonal projections. Observe that 
\begin{equation}\label{eq: igualdad 1}
    H_\tau^{\widehat{\bot}}=H_\sigma^\bot \quad\mathrm{and} \quad
P_{H_\tau^{\widehat{\bot}}}\widehat{K}=P_{H_\tau^{\widehat{\bot}}}\left(P_{H_\sigma^\bot\oplus H_\tau}K\right)=P_{H_\sigma^\bot} K.
\end{equation}
 Similarly, 
 \begin{equation}\label{eq: igualdad 2}
     H_{\tau_j}^{\widehat{\bot}}=H_\sigma^\bot\oplus H_{\tau\setminus\tau_j}\quad\mathrm{and}\quad P_{H_{\tau_j}^{\widehat{\bot}}}\widehat{K}=P_{H_{\tau_j}^{\widehat{\bot}}}\left(P_{H_\sigma^\bot\oplus H_\tau}K\right)=P_{H_\sigma^\bot\oplus H_{\tau\setminus\tau_j}} K.
 \end{equation}
Therefore, dividing both sides of \eqref{eq: local para el cuerpo proyectado} by $\vol{P_{H_\sigma^\bot}K}^m$ and using \eqref{eq: igualdad 1} and \eqref{eq: igualdad 2} we get that
\begin{equation*}\label{}
   \left({n-|\sigma|+|\tau|\choose n-|\sigma|}\frac{\vol{P_{H_\sigma^\bot\oplus H_\tau}K}}{\vol{P_{H_\sigma^\bot}K}}\right)^{m-s} \leq \prod_{j=1}^m{n-|\sigma|+|\tau|-|\tau_j|\choose n-|\sigma|}\frac{\vol{P_{H_{\sigma}^\bot\oplus H_{\tau\setminus\tau_j}}K}}{\vol{P_{H_\sigma^\bot}K}}. 
\end{equation*}
In particular, this can be applied to the $(|\tau|-1)$-cover of $\tau$ given by $(\tau\setminus\{j\}:j\in\tau)$, thus
\begin{equation}\label{eq: ineq 3 applied to tau}
   {n-|\sigma|+|\tau|\choose n-|\sigma|}\frac{\vol{P_{H_\sigma^\bot\oplus H_\tau}K}}{\vol{P_{H_\sigma^\bot}K}} \leq \prod_{j\in\tau}{n-|\sigma|+1\choose n-|\sigma|}\frac{\vol{P_{H_\sigma^\bot\oplus H_{\{j\}}}K}}{\vol{P_{H_\sigma^\bot}K}}. 
\end{equation}

Inequality \eqref{eq: ineq 2 localBT} can be applied to irreducible $s$-covers of $\sigma$. Hence, if we take all the irreducible $s$-covers of $\sigma,$ we have a finite number of inequalities of the type \eqref{eq: ineq 2 localBT} $(\sigma_1,\ldots,\sigma_m)$, is irreducible. Therefore, we have a finite number of inequalities
involving all the elements of the set $$\left\{{n-|\tau|\choose n-|\sigma|}\frac{\vol{P_{H_{\sigma}^\bot\oplus H_{\sigma\setminus\tau}}K}}{\vol{P_{H_\sigma^\bot}K}}:\tau\subseteq\sigma\right\}.$$
Now, we consider 
the set of 
positive real numbers $\{(x_\tau)_{\tau\subseteq\sigma}:x_\tau>0,\,\forall\,\tau\subseteq\sigma\}$ such that 
\begin{equation}\label{eq: cond 1 min problem}
    x_\sigma:={n\choose n-|\sigma|}\frac{\vol{K}}{\vol{P_{H_\sigma^\bot}K}},
\end{equation}
for any $\emptyset\neq\tau\subseteq\sigma$
\begin{equation}\label{eq: cond 2 min problem}
    x_{\sigma\setminus\tau}\leq {n-|\tau|\choose n-|\sigma|}\frac{\vol{P_{H_\sigma^\bot\oplus H_{\sigma\setminus\tau}}K}}{\vol{P_{H_\sigma^\bot}K}}
\end{equation}
and satisfying the following inequalities for any irreducible $s$-cover of $\sigma$, $(\sigma_1,\dots,\sigma_m)$,
 \begin{equation}\label{eq: cond 3 min problem}
     x_\sigma^{m-s} \leq \prod_{j=1}^m x_{\sigma\setminus\sigma_j},
 \end{equation}
referred as Type-$(1)$ inequalities,
and 
\begin{equation}\label{eq: cond 4 min problem}
    x_\tau\leq\prod_{j\in\tau}x_{\{j\}},
\end{equation}
referred as Type-$(2)$ inequalities. We observe that Type-$(1)$ and Type-$(2)$ inequalities are inspired by \eqref{eq: ineq 2 localBT} and \eqref{eq: ineq 3 applied to tau}, respectively.

In view of all the conditions above, there exists a tuple $(t_\tau)_{\tau\subseteq\sigma}$ which is minimal fulfilling the conditions \eqref{eq: cond 1 min problem}, \eqref{eq: cond 2 min problem}, \eqref{eq: cond 3 min problem} and \eqref{eq: cond 4 min problem} (see Lemma \ref{lem:minimalSet} for a detailed proof of this fact).

Moreover, by the minimality of $(t_\tau)_{\tau\subseteq\sigma}$, for every $i\in\sigma$, 
there is an inequality of Type-$(1)$ or Type-$(2)$ involving $t_{\{i\}}$ which is an equality.
If for some $i\in \sigma$ there is equality in an inequality of Type-$(1)$ this means that there exists $\beta(i):=(\sigma_1^i,\ldots,\sigma_{m_i}^i)$ an irreducible $s_i$-cover of $\sigma$ with $\sigma^i_{j_0}=\sigma\setminus\{i\}$ for some $j_0\in\{1,\ldots,m_i\}$ and such that
$$t_\sigma^{m_i-s_i}=\prod_{j=1}^{m_i}t_{\sigma\setminus\sigma^i_j}=t_{\{i\}}\prod_{j\in [m_i]\setminus\{j_0\}}t_{\sigma\setminus\sigma^i_j}.$$
If for some $i\in\sigma$ there is equality in an inequality of Type-$(2)$ this means that there exists a subset $\tau(i) $ of $\sigma$ such that $i\in\tau(i)$ and 
\begin{equation}\label{eq: equality type 2}
    t_{\tau(i)}=\prod_{j\in\tau(i)}t_{\{j\}}.
\end{equation}
Moreover, by the minimality of $(t_\tau)_{\tau\subseteq\sigma}$
%$t_{\tau(i)}$ 
there exists $(\sigma^i_1,\ldots,\sigma^i_{m_i})$ an irreducible $s_i$-cover of $\sigma$ such that $\sigma^i_{j_0}=\sigma\setminus\tau(i)$ and with equality in an inequality of Type-$(1)$. Otherwise we 
could replace $t_{\tau(i)}$ by a strictly smaller value
verifying all the conditions 
%of the minimality process 
posted above. Therefore,
\begin{eqnarray}\label{eq: equality type 2 reduction}
    t_\sigma^{m_i-s_i}&=&\prod_{j=1}^{m_i}t_{\sigma\setminus\sigma^i_j}=t_{\tau(i)}\prod_{j\in [m_i]\setminus\{j_0\}}t_{\sigma\setminus\sigma^i_j}\nonumber\\
    &=&\left(\prod_{j\in\tau(i)}t_{\{j\}}\right)\left(\prod_{j\in [m_i]\setminus\{j_0\}}t_{\sigma\setminus\sigma^i_j}\right),
\end{eqnarray}
where in the last equality we have used \eqref{eq: equality type 2}. Now, we notice that taking complementaries in $\sigma$ of each element of a $s$-cover of $\sigma$, with $m$ sets, transforms it into a $(m-s)$-cover. Thus, $(\sigma\setminus\sigma^i_1,\ldots,\sigma\setminus\sigma^i_{m_i})$ is a $(m_i-s_i)$-cover of $\sigma$. Moreover, since $\sigma\setminus\sigma^i_{j_0}=\tau(i)$, we have that $(\sigma\setminus\sigma^i_j:j\in[m_i]\setminus\{j_0\})\cup(\{j\}:j\in\tau(i))$ is also a $(m_i-s_i)$-cover of $\sigma$ with $m_i-1+|\tau(i)|$ sets. Therefore, $\beta(i):=(\sigma^i_j:j\in[m_i]\setminus\{j_0\})\cup(\sigma\setminus\{j\}:j\in\tau(i))$ is a $(s_i+|\tau(i)|-1)$-cover of $\sigma$ with equality in the inequality of Type-$(1)$ by \eqref{eq: equality type 2 reduction}.

Now we define $\beta:=\cup_{i\in\sigma}\beta(i)$. It is  a $s$-cover of $\sigma$ with 
\[
s=\sum_{i\in\mathcal{M}_1}s_i+\sum_{i\in\mathcal{M}_2}(s_i+|\tau(i)|-1),
\]
where $\mathcal{M}_1:=\{i\in\sigma: \text{ there exists }\,(\sigma^i_1,\ldots,\sigma^i_{m_i}) \text{ an irreducible }s_i\text{-cover of }\sigma \text{ with } \sigma^i_{j_0}=\sigma\setminus\{i\} \text{ for some } j_0\in[m_i]\text{ and equality in a Type-}(1) \text{ inequality}\}$
and $\mathcal{M}_2:=\sigma\setminus \mathcal{M}_1.$ The number of elements of $\beta$ is
\[
m=\sum_{i\in\mathcal{M}_1}m_i+\sum_{i\in\mathcal{M}_2}(m_i-1+|\tau(i)|).
\]
Then, we have that the $s$-cover of $\sigma$ given by $\beta$ also has equality in the inequality of Type-$(1)$ applied to $\beta$ since
\begin{eqnarray*}
    t_\sigma^{m-s}&=&t_\sigma^{\sum_{i\in\mathcal{M}_1}(m_i-s_i)}t_\sigma^{\sum_{i\in\mathcal{M}_2}((m_i-1+|\tau(i)|)-(s_i+|\tau(i)|-1))}\nonumber\\
    &=&\left[\prod_{i\in\mathcal{M}_1}t_\sigma^{m_i-s_i}\right]\left[\prod_{i\in\mathcal{M}_2}t_\sigma^{m_i-s_i}\right]\nonumber\\
    &=&\left[\prod_{i\in\mathcal{M}_1}\left(\prod_{\tau\in\beta(i)}t_{\sigma\setminus\tau}\right)\right]\left[\prod_{i\in\mathcal{M}_2}\left(\prod_{\tau\in\beta(i)}t_{\sigma\setminus\tau}\right)\right]\nonumber\\
    &=&\prod_{\tau\in\beta}t_{\sigma\setminus\tau}.
\end{eqnarray*}
Considering $\beta'=\beta\setminus(\sigma\setminus\{i\}:i\in\sigma)$, which is a $(s-|\sigma|+1)$-cover of $\sigma$ with $m-|\sigma|$ elements, we can rewrite the equality above as
\begin{equation}\label{eq: eq equality beta}
    t_\sigma^{m-s}=\left(\prod_{\tau\in\beta'}t_{\sigma\setminus\tau}\right)\left(\prod_{i\in\sigma}t_{\{i\}}\right).
\end{equation}
Besides, applying inequality of Type-$(1)$ to the cover $\beta'$ (it can be also applied to reducible covers with the same steps done in \eqref{eq: eq equality beta} but with inequality instead of equality) 
we get that
\begin{equation}\label{eq: ineq beta'}
  t_\sigma^{m-s-1}\leq\prod_{\tau\in\beta'}t_{\sigma\setminus\tau}. 
\end{equation}
Therefore, from \eqref{eq: eq equality beta} and \eqref{eq: ineq beta'} it follows that
\[
t_\sigma\geq\prod_{i\in\sigma}t_{\{i\}}.
\]
On the other hand, taking $\tau=\sigma$ in the inequality of Type-$(2)$ we have the reverse inequality, $t_\sigma\leq\prod_{i\in\sigma}t_{\{i\}}$. Therefore, 
\[
t_\sigma=\prod_{i\in\sigma}t_{\{i\}}.
\]

Let now $\tau\subset\sigma$. Since $\tau\cup(\{i\}:i\in\sigma\setminus\tau)$ is a $1$-cover of $\sigma$ with $|\sigma|-|\tau|+1$ sets, we have that $(\sigma\setminus\tau)\cup(\sigma\setminus\{i\}:i\in\sigma\setminus\tau)$ is a $(|\sigma|-|\tau|)$-cover of $\sigma$. Then,
\[
t_\sigma\leq t_\tau\left(\prod_{i\in\sigma\setminus\tau}t_{\{i\}}\right)\leq\left(\prod_{i\in\tau}t_{\{i\}}\right)\left(\prod_{i\in\sigma\setminus\tau}t_{\{i\}}\right)=\prod_{i\in\sigma}t_{\{i\}}=t_\sigma,
\]
where we have used the inequalities of Type-$(1)$ and Type-$(2)$ in the first and the second inequalities above, respectively.
Finally, we also deduce that \[
t_\tau=\prod_{i\in\tau}t_{\{i\}}.
\]

Let us now consider the convex body given by
\[
B_K:=\mathrm{conv}\left(\left\{\sum_{i\in\sigma}\frac{t_{\{i\}}}{2}[-e_i,e_i],L\right\}\right),
\]
where $L$ is any $(n-|\sigma|)$-dimensional convex body contained in $H_\sigma^\bot$, with $0\in L$ and $\vol{L}=\vol{P_{H_\sigma^\bot}K}$.
   
Notice that
\begin{eqnarray*}
    \vol{B_K}&=&\frac{1}{{n \choose n-|\sigma|}}\left(\prod_{i\in\sigma}t_{\{i\}}\right)\vol{P_{H_\sigma^\bot}K}\\
    &=&\frac{1}{{n \choose n-|\sigma|}}t_\sigma\vol{P_{H_\sigma^\bot}K}\\
    &=&\vol{K},
\end{eqnarray*}
where above we have used equality case of \cite[Thm.~41]{RS58}. Notice also that for any $\tau\subseteq\sigma$ we have that
\begin{eqnarray*}
    \vol{P_{H_\sigma^\bot\oplus H_\tau}B_K}&=&\frac{1}{{n-|\sigma|+|\tau|\choose n-|\sigma|}}\left(\prod_{i\in\tau}t_{\{i\}}\right)\vol{P_{H_\sigma^\bot}K}\\
    &=&\frac{1}{{n-|\sigma|+|\tau|\choose n-|\sigma|}}t_\tau\vol{P_{H_\sigma^\bot}K}\\
    &\leq&\frac{1}{{n-|\sigma|+|\tau|\choose n-|\sigma|}}{n-|\sigma|+|\tau|\choose n-|\sigma|}\frac{\vol{P_{H_\sigma^\bot\oplus H_{\tau}}K}}{\vol{P_{H_\sigma^\bot}K}}\vol{P_{H_\sigma^\bot}K}\\
    &=&\vol{P_{H_\sigma^\bot\oplus H_\tau}K}.
\end{eqnarray*}
Concluding the proof of the result.
\end{proof}

As a consequence of Theorem \ref{thm: isoperimetric main projections} we show an isoperimetric type result for a suitable Hanner polytope.
\begin{cor}\label{cor: Hanner}
Let $K\in\mathcal{K}^n$ and $\sigma\subset [n]$. Then, there exist some $c_i>0$, for $i\in\sigma$, and a Hanner polytope $B_K$ of the form
    $$B_K=\mathrm{conv}\left(\left\{\sum_{i\in\sigma}c_i[-e_i,e_i],
    \sum_{i\notin\sigma}b_i[-e_i,e_i]
    \right\}\right)
    $$
with $b_i>0$ for every $i\in[n]\setminus\sigma$, such that
    \begin{itemize}
        \item[i)] $\vol{K}=\vol{B_K}$,
        \item[ii)] $\vol{P_{H_\sigma^\bot\oplus H_\tau }K}\geq\vol{P_{H_\sigma^\bot\oplus H_\tau }B_K}$ for every $\tau\subseteq \sigma$ and
        \item[iii)] $2^{n-|\sigma|}\prod_{i\notin\sigma}b_i=\vol{P_{H_\sigma^\bot}K}$.
    \end{itemize}
\end{cor}

The following lemma assures the existence of such set $(t_\tau)_{\tau\subseteq\sigma
}$ of minimal positive real numbers satisfying the conditions of the problem used in the proof of Theorem \ref{thm: isoperimetric main projections}.

\begin{lemma}\label{lem:minimalSet}
Let $K\in\mathcal{K}^n$ and $\sigma\subset[n]$. There exists a set of minimal positive real numbers $\{t_\tau\in[C(K,\sigma),D(K,\sigma)]:\tau\subseteq\sigma\}$ satisfying the conditions \eqref{eq: cond 1 min problem}, \eqref{eq: cond 2 min problem}, \eqref{eq: cond 3 min problem} and \eqref{eq: cond 4 min problem}, where $C(K,\sigma)$ and $D(K,\sigma)$ are two absolute positive real constants.
\end{lemma}
\begin{proof}
Let us consider all the sets of positive real numbers indexed by $\tau\subseteq\sigma$ under the conditions  \eqref{eq: cond 1 min problem}, \eqref{eq: cond 2 min problem}, \eqref{eq: cond 3 min problem} and \eqref{eq: cond 4 min problem}, i.e.
\[
\begin{split}
    \mathcal{C}:=\{(x_\tau)_{\tau\subseteq\sigma}:x_\tau> 0,\,\,\forall\,\tau\subseteq\sigma, \,\text{ fulfilling \eqref{eq: cond 1 min problem}, \eqref{eq: cond 2 min problem}, \eqref{eq: cond 3 min problem}, \eqref{eq: cond 4 min problem}}\}
\end{split}
\]
Observe that by \eqref{eq: cond 1 min problem} $$x_\sigma={n
    \choose n-|\sigma|}\frac{\vol{K}}{\vol{P_{H_\sigma^\bot}K}}$$ for every $(x_{\tau})_{\tau\subseteq\sigma}\in\mathcal{C}$.
This set is not empty since precisely due to \eqref{eq: ineq 2 localBT}, \eqref{eq: ineq 3 applied to tau} the tuple 
\[
\left\{d_{\sigma\setminus\tau}(K):={n-|\tau|
    \choose n-|\sigma|}\frac{\vol{P_{H_\sigma^\bot\oplus H_{\sigma\setminus\tau}}K}}{\vol{P_{H_\sigma^\bot}K}}: \tau\subseteq\sigma\right\} 
\]
belongs to $\mathcal{C}.$ 

Now we are considering the following order relation. Given $(x_{\tau})_{\tau\subseteq\sigma}, (y_{\tau})_{\tau\subseteq\sigma}\in\mathcal{C}$, let us define the order 
\[
(x_\tau)_{\tau\subseteq\sigma}\prec (y_\tau)_{\tau\subseteq\sigma} \quad :\Leftrightarrow \quad x_\tau\leq y_\tau\text{  for all }\tau\subseteq\sigma.
\]
Then, $(\mathcal{C},\prec)$ is a (partially) ordered set. 

Let us notice that from \eqref{eq: cond 2 min problem} we have that 
\[
x_\tau \leq d_\tau(K)={n-|\sigma|+|\tau|\choose n-|\sigma|}\frac{\vol{P_{H_\sigma^\bot\oplus H_{\tau}}K}}{\vol{P_{H_\sigma^\bot}K}},
\]
for every $\tau\subseteq\sigma$, and let $D(K,\sigma):=\max\{d_\tau(K):\tau\subseteq\sigma\}$. For a fixed $\tau\subseteq\sigma$, consider the irreducible $1$-cover $\{\sigma\setminus\tau\}\cup\{i:i\in\tau\}$ of $\sigma$. By \eqref{eq: cond 3 min problem} we obtain that
\[
x_\sigma^{|\tau|}=x_\sigma^{|\tau|+1-1} \leq x_\tau\left(\prod_{i\in\tau}x_{\sigma\setminus\{i\}}\right),
\]
for any $(x_\tau)_{\tau\subseteq\sigma}\in\mathcal{C}.$
Then
\[
x_\tau \geq \frac{x_\sigma^{|\tau|}}{\prod_{i\in\tau}x_{\sigma\setminus\{i\}}} \geq 
\frac{x_\sigma^{|\tau|}}{\prod_{i\in\tau}d_{\sigma\setminus\{i\}}(K)} =: c_\tau(K).
\]
Finally, let $C(K,\sigma):=\min\{c_\tau(K):\tau\subseteq\sigma\}$. Then for every $(x_\tau)_{\tau\subseteq\sigma}\in\mathcal C$ we have that $C(K,\sigma) \leq x_\tau \leq D(K,\sigma)$ for every $\tau\subseteq\sigma$. In particular, every decreasing chain $\{(x^\ell_\tau)_{\tau\subseteq\sigma}\}_{\ell\in\mathbb{N}}\subset\mathcal{C}$ is lower bounded 
and hence by Zorn's Lemma there exists a minimal element in $\mathcal C$.
\end{proof}

\begin{rmk}
    Notice that in Lemma \ref{lem:minimalSet} we may have omitted in the argument the restriction to irreducible covers of $\sigma$.
\end{rmk}

\section{Characterization of equality case of local Bollob\'as-Thomason inequalities}\label{sec: characterization equality case}

The Brunn-Minkowski inequality \cite{AGM15} states that for every $K,C\in\mathcal K^n$, and every $\lambda\in[0,1]$, it holds that
\begin{equation}\label{eq:BrunnMinkowski}
\vol{(1-\lambda)K+\lambda C}^\frac1n \geq (1-\lambda)\vol{K}^\frac1n+\lambda\vol{C}^\frac1n.
\end{equation}
Moreover, equality holds if and only either $K$ and $C$ are rescales of each other or $K$ and $C$ are contained in parallel hyperplanes.

For every $f:\mathbb R^n\rightarrow\mathbb{R}$, let $\mathrm{hyp}(f):=\{(x,t):t\leq f(x)\}$ be the hypograph of $f$. We say that $f$ is $1/k$-concave if $f^{\frac1k}$ is a concave function with $k>0$. Recall that a function is concave if and only if its hypograph is a convex set. Moreover, for every $K\in\mathcal{K}^n$ containing the origin and $x\in\mathbb R^n$ we define the Minkowski functional of $x$ respect to $K$ as $\Vert x\Vert_K:=\min\{\lambda\geq 0:x\in\lambda K\}$. The set of all subspaces of dimension $i\in[n]$ in $\mathbb R^n$ will be denoted by  $\mathcal L^n_i$.

We need the following lemma.
\begin{lemma}\label{lem:conicalEpi}
Let $K\in\mathcal K^n$, $H\in\mathcal L^n_i$
and let $f:P_HK\rightarrow[0,\infty)$ be the concave function given by
$$f(y):=\vol{K\cap\left(y+H^\bot\right)}^\frac{1}{n-i}.$$
Let $y_0\in P_HK$ be such that
$$
\mathrm{hyp}(f)=\mathrm{conv}(P_HK\times\{0\},y_0\times\|f\|_\infty).
$$
Then for every $y\in P_HK$ we have that
\[
K\cap\left(y+H^\bot\right)=z(y)+(1-\|y-y_0\|_{P_HK-y_0})\left(K\cap\left(y_0+H^\bot\right)\right),
\]
for some $z(y)\in \mathbb R^n$.
\end{lemma}

\begin{proof}
    First of all, for $y\neq y_0$, $y\in P_HK$, we can write it as
    \begin{equation}\label{eq: convex combination y}
    y=\|y-y_0\|_{P_HK-y_0}\left(y_0+\frac{y-y_0}{\|y-y_0\|_{P_HK-y_0}}\right)+(1-\|y-y_0\|_{P_HK-y_0})y_0,
     \end{equation}
   which is a convex combination of $y_0$ and 
    \[
    y_0+\frac{y-y_0}{\|y-y_0\|_{P_HK-y_0}}\in\mathrm{relbd}(P_HK),
    \]
    since $0<\|y-y_0\|_{P_HK-y_0}\leq 1$. Before, $\mathrm{relbd}(P_HK)$ denotes the relative boundary of $P_HK$. Then, by the concavity of $f$ it follows that
   \[
    \begin{split}
            f(y) & \geq\|y-y_0\|_{P_HK-y_0}f\left(y_0+\frac{y-y_0}{\|y-y_0\|_{P_HK-y_0}}\right)+(1-\|y-y_0\|_{P_HK-y_0})f(y_0) \\
            & =(1-\|y-y_0\|_{P_HK-y_0})\|f\|_\infty,
    \end{split}
    \]
    where in the last equality we have used that
   \[
   f\left(y_0+\frac{y-y_0}{\|y-y_0\|_{P_HK-y_0}}\right)=0.
   \]
    Now, let us consider $t\geq 0$ be such that $(y,t)\in\mathrm{hyp}(f)$. Then, there exists $x\in P_H K$ and $\mu\in[0,1]$ such that
   \[(y,t)=\mu(x,0)+(1-\mu)(y_0,\|f\|_\infty),\]
   so $y=\mu x+(1-\mu) y_0 $ and $t=(1-\mu) \|f\|_\infty.$ Therefore, using \eqref{eq: convex combination y} we get that
   \[
   \mu x+(1-\mu) y_0= \|y-y_0\|_{P_HK-y_0}\left(y_0+\frac{y-y_0}{\|y-y_0\|_{P_HK-y_0}}\right)+(1-\|y-y_0\|_{P_HK-y_0})y_0,
   \]
   that can be rewritten as
   \[
   \mu( x- y_0)= \|y-y_0\|_{P_HK-y_0}\frac{y-y_0}{\|y-y_0\|_{P_HK-y_0}}.
   \]
  Assume that $\mu<\|y-y_0\|_{P_HK-y_0}$. Then,
   \[
   x-y_0=\frac{\|y-y_0\|_{P_HK-y_0}}{\mu}\frac{y-y_0}{\|y-y_0\|_{P_HK-y_0}}
   \]
   and so
   \[
   \| x-y_0\|_{P_HK-y_0}=\left\|\frac{\|y-y_0\|_{P_HK-y_0}}{\mu}\frac{y-y_0}{\|y-y_0\|_{P_HK-y_0}}\right\|_{P_HK-y_0}=\frac{\|y-y_0\|_{P_HK-y_0}}{\mu}> 1,
   \]
   which contradicts that $x\in P_HK$. Therefore, $\mu\geq\|y-y_0\|_{P_HK-y_0}$ and we get that
   \[
   t=(1-\mu) \|f\|_\infty\leq (1-\|y-y_0\|_{P_HK-y_0}) \|f\|_\infty
   \]
   for every $t\in\{t:(y,t)\in\mathrm{hyp}(f)\}$. Then, $f(y)=\sup\{t:(y,t)\in\mathrm{hyp}(f)\}\leq (1-\|y-y_0\|_{P_HK-y_0}) \|f\|_\infty $. Therefore, $f(y)=(1-\|y-y_0\|_{P_HK-y_0}) \|f\|_\infty$.
   That means
    \begin{equation}\label{eq:equalCone}
    \vol{K\cap\left(y+H^\bot\right)}^\frac{1}{n-i} = (1-\|y-y_0\|_{P_HK-y_0}) \vol{K\cap\left(y_0+H^\bot\right)}^\frac{1}{n-i}
    \end{equation}
    for every $y\in P_HK$. Notice that by the convexity of $K$ we have that
    \[
    \begin{split}
        K\cap\left(y+H^\bot\right) \supseteq \|y-y_0\|_{P_HK-y_0} \left(K\cap\left(y_0+\frac{y-y_0}{\|y-y_0\|_{P_HK-y_0}}+H^\bot\right)\right) & \\
        +(1-\|y-y_0\|_{P_HK-y_0})\left(K\cap\left(y_0+H^\bot\right)\right). &
    \end{split}
    \]
    By \eqref{eq:BrunnMinkowski} we get that
    \[
    \begin{split}
        \vol{K\cap\left(y+H^\bot\right)}^\frac{1}{n-i} & \geq  \|y-y_0\|_{P_HK-y_0} \mathrm{vol}\left(K\cap\left(y_0+\frac{y-y_0}{\|y-y_0\|_{P_HK-y_0}}+H^\bot\right)\right)^\frac{1}{n-i} \\
       & +(1-\|y-y_0\|_{P_HK-y_0})\vol{K\cap\left(y_0+H^\bot\right)}^\frac{1}{n-i} \\
       & = (1-\|y-y_0\|_{P_HK-y_0})\vol{K\cap\left(y_0+H^\bot\right)}^\frac{1}{n-i}.
    \end{split}
    \]
    This together with \eqref{eq:equalCone} implies equality in Brunn-Minkowski inequality above. Thus
    \[
    K\cap \left(y+H^\bot\right) = z(y)+(1-\|y-y_0\|_{P_HK-y_0})\left(K\cap\left(y_0+H^\bot\right)\right)
    \]
    for every $y\in P_HK$ and for some $z(y)\in \mathbb R^n$.
\end{proof}

Berwald inequality \cite{B47} is a kind of reverse inequality to H\"older inequality for concave functions. It states that for every $0<\gamma_1<\gamma_2$, $K\in\mathcal K^n$, $f:K\rightarrow[0,\infty)$ concave, continuous and not identically null, then
    \begin{equation}\label{eq:Berwald}
    \left(\frac{{\gamma_2+n \choose n}}{\mathrm{vol}(K)}\int_Kf(x)^{\gamma_2} dx\right)^{\frac{1}{\gamma_2}}
    \leq \left(\frac{{\gamma_1+n \choose n}}{\mathrm{vol}(K)}\int_Kf(x)^{\gamma_1} dx\right)^{\frac{1}{\gamma_1}}.
    \end{equation}
Moreover, equality holds above if and only if $\mathrm{hyp}(f)=\mathrm{conv}(K\times\{0\},y_0\times\|f\|_\infty)$, for some $y_0\in\mathrm{supp}(f)$.

\begin{proof}[Proof of Theorem \ref{thm:LocalBollobasThomasonineq}]
The proof is based on a detailed analysis of the proof of the inequality \eqref{eq:ineqProj}, paying attention to the conditions under which inequalities become equalities. For that reason and for the readers' convenience we include in Proposition \ref{prop:localBollobasThomason} of the Appendix \ref{Sec:Appendix} the proof of \eqref{eq:ineqProj} as given in \cite[Thm.~4]{MNZ24} by Manui, Ndiaye and Zvavitch.  

In order to have equality in \eqref{eq:ineqProj}, we must have equality in \eqref{(1)}, \eqref{(2)}, and \eqref{(3)}.

Equality holds in \eqref{(1)} if and only if for almost every $y\in P_{H_\sigma^\bot}K$ we have that
\[
\begin{split}
K\cap(y+H_\sigma) & = z_1(y) + \sum_{j=1}^k P_{y+H_{\overline{\sigma}_j}}(K\cap(y+H_\sigma)) \\
& = z_1(y)+\sum_{j=1}^k P_{H_{\overline{\sigma}_j}\oplus H_\sigma^\bot}(K)\cap\left(y+H_{\overline{\sigma}_j}\right). 
\end{split}
\]
where $(\overline{\sigma}_1,\dots,\overline{\sigma}_k)$ denotes the induced cover of $(\sigma\setminus\sigma_1,\dots,\sigma\setminus\sigma_m)$ and where $z_1(y)\in\mathbb R^n$, see equality case of \eqref{eq: Bollobas-Thomason}. 
Indeed, by the continuity of the arguments, we get the identity above for every $y\in \mathrm{relint}(P_{H_\sigma^\bot}K)$, and again using the continuity we get equality for every $y\in P_{H_\sigma^\bot}K$, proving \eqref{eq:Charact(1)}.

Equality in \eqref{(2)} holds if and only if there exist $\alpha_1,\dots,\alpha_m>0$ such that
\[
\alpha_1 f_1(y)^{\frac{|\sigma|-|\sigma_1|}{m-s}p_1} = \cdots = \alpha_mf_m(y)^{\frac{|\sigma|-|\sigma_m|}{m-s}p_m}
\]
i.e.
\begin{equation}\label{eq:ConseOf(2)}
\begin{split}
 & \alpha_1\vol{P_{H_{\sigma\setminus\sigma_1}}(K\cap(y+H_\sigma))}^\frac{|\sigma|}{|\sigma|-|\sigma_1|} \\
 & =\cdots=
\alpha_m\vol{P_{H_{\sigma\setminus\sigma_m}}(K\cap(y+H_\sigma))}^\frac{|\sigma|}{|\sigma|-|\sigma_m|}.
\end{split}
\end{equation}
In particular, there exists $y_0\in P_{H_\sigma^\bot}K$ such that $f_i(y_0)=\|f_i\|_\infty$, for every $i=1,\dots,m$. 

Equality holds in \eqref{(3)} if and only if 
\[
\mathrm{hyp}(f_i)=\mathrm{conv}\left(P_{H_\sigma^\bot}K\times\{0\},y_0\times \|f_i\|_\infty\right),
\]
for every $i=1,\dots,m$. Notice that 
\begin{equation}\label{eq:Prop1}
P_{H_{\sigma\setminus\sigma_i}}(K\cap(y+H_\sigma)) = z_3(y,i)+P_{H_{\sigma\setminus\sigma_i}\oplus H_\sigma^\bot}(K)\cap\left(y+H_{\sigma\setminus\sigma_i}\right),
\end{equation}
for some $z_3(y,i)\in H_\sigma^\bot$. Thus
\[
f_i(y)=\vol{P_{H_{\sigma\setminus\sigma_i}\oplus H_\sigma^\bot}(K)\cap\left(y+H_{\sigma\setminus\sigma_i}\right)}^\frac{1}{|\sigma|-|\sigma_i|}.
\]
Since $P_{H_{\sigma\setminus\sigma_i}^\bot}\left(P_{H_{\sigma\setminus\sigma_i}\oplus H_\sigma^\bot}\right)=P_{H_\sigma^\bot}$, by Lemma \ref{lem:conicalEpi} we obtain that
\[
\begin{split}
& P_{H_{\sigma\setminus\sigma_i}\oplus H_\sigma^\bot}(K)\cap\left(y+H_{\sigma\setminus\sigma_i}\right)  \\
& = z_4(y,i)+\left(1-\|y-y_0\|_{P_{H_\sigma^\bot}K-y_0}\right)\left(P_{H_{\sigma\setminus\sigma_i}\oplus H_\sigma^\bot}(K)\cap\left(y_0+H_{\sigma\setminus\sigma_i}\right)\right).
\end{split}
\]
Now notice that the induced cover $(\overline{\sigma}_1,\dots,\overline{\sigma}_k)$ fulfills $\overline{\sigma}_j\subset \sigma_i$ or $\overline{\sigma}_j\subset \sigma\setminus\sigma_i$, for every $i=1,\dots,m$ and every $j=1,\dots,k$. Hence
\begin{equation*}
\begin{split}
 P_{H_{\sigma\setminus\sigma_i} \oplus H_\sigma^\bot} & (K)\cap\left(y+H_{\sigma\setminus\sigma_i}\right) \\
& = -z_3(y,i)+P_{H_{\sigma\setminus\sigma_i}}(K\cap(y+H_\sigma)) \\
& = 
-z_3(y,i)+P_{H_{\sigma\setminus\sigma_i}} \left(z_1(y)+\sum_{j=1}^k P_{H_{\overline{\sigma}_j}\oplus H_\sigma^\bot}(K)\cap\left(y+H_{\overline{\sigma}_j}\right)\right) \\
& = z_5(y,i) + \sum_{j=1}^k P_{H_{\sigma\setminus\sigma_i}}\left(P_{H_{\overline{\sigma}_j}\oplus H_\sigma^\bot}(K)\cap\left(y+H_{\overline{\sigma}_j}\right)\right) \\
& = z_5(y,i) + \sum_{j\in M_i}P_{H_{\overline{\sigma}_j}\oplus H_\sigma^\bot}(K)\cap\left(y+H_{\overline{\sigma}_j}\right),
\end{split}
\end{equation*}
where $M_i:=\{j\in[k]:\overline{\sigma}_j\subset \sigma\setminus\sigma_i\}$. Mixing the last two equations above we obtain
\begin{equation}\label{eq:Prop2}
\begin{split}
& \sum_{j\in M_i}P_{H_{\overline{\sigma}_j}\oplus H_\sigma^\bot}(K)\cap\left(y+H_{\overline{\sigma}_j}\right) \\
& = z_2(y,i) + \left(1-\|y-y_0\|_{P_{H_\sigma^\bot}K-y_0}\right)\left( P_{H_{\sigma\setminus\sigma_i}\oplus H_\sigma^\bot}(K)\cap\left(y_0+H_{\sigma\setminus\sigma_i}\right)\right),
\end{split}
\end{equation}
where $z_2(y,i)=z_4(y,i)-z_5(y,i)$, which shows \eqref{eq:Charact(2)}.

Finally, \eqref{eq:Charact(3)} is simply a consequence of \eqref{eq:Prop1},\eqref{eq:Prop2} together with \eqref{eq:ConseOf(2)}, because
\[
\begin{split}
\mathrm{vol} & \left(\sum_{j\in M_i}P_{H_{\overline{\sigma}_j}\oplus H_\sigma^\bot}(K)\cap\left(y+H_{\overline{\sigma}_j}\right)\right)^\frac{1}{|\sigma|-|\sigma_i|} \\
& = 
L_i\mathrm{vol}\left(\sum_{j\in M_1}P_{H_{\overline{\sigma}_j}\oplus H_\sigma^\bot}(K)\cap\left(y+H_{\overline{\sigma}_j}\right)\right)^\frac{1}{|\sigma|-|\sigma_1|},
\end{split}
\]
and where $L_i:=\left(\frac{\alpha_1}{\alpha_i}\right)^{\frac{1}{|\sigma|}}$, concluding \eqref{eq:Charact(3)}.
\end{proof}

\section*{Acknowledgments}
The authors would like to thank the anonymous referee for his/her valuable suggestions and useful comments which contributed to improve the paper.

\section{Appendix}\label{Sec:Appendix}

The following Proposition was proven by Manui, Ndiaye, and Zvavitch in \cite[Thm.~4]{MNZ24}. For a different proof, see \cite[Thm.~1.5]{ABBC21}.

\begin{propo}\label{prop:localBollobasThomason}
    Let $\sigma\subset[n]$, and let $(\sigma_1,\dots,\sigma_m)$ be a $s$-cover of $\sigma$. For every $K\in\mathcal K^n$ we have that
    \begin{equation*}
    \mathrm{vol}(K)^{m-s}\vol{P_{H_\sigma^\bot}K}^s \leq \frac{\prod_{i=1}^m{n-|\sigma_i|\choose n-|\sigma|}}{{n\choose n-|\sigma|}^{m-s}}\prod_{i=1}^m\vol{P_{H_{\sigma_i}^\bot}K}.
    \end{equation*}

\end{propo}

\begin{proof}
    For every $y\in P_{H_\sigma^\bot}K$, we apply Bollob\'as-Thomason inequality \eqref{eq: Bollobas-Thomason} to $K\cap(y+H_\sigma)$ and to the $(m-s)$-cover $(\sigma\setminus\sigma_1,\dots,\sigma\setminus\sigma_m)$, obtaining
    \[
    \vol {K\cap(y+H_\sigma)} \leq \prod_{i=1}^m\vol{P_{H_{\sigma\setminus\sigma_i}}(K\cap(y+H_\sigma))}^\frac{1}{m-s}.
    \]
    Thus 
    \begin{equation}\label{(1)}
    \begin{split}
        \mathrm{vol}(K) & = \int_{P_{H_\sigma^\bot}K}\vol{K\cap(y+H_\sigma)}dy \\
        & \leq \int_{P_{H_\sigma^\bot}K} \prod_{i=1}^m\vol{P_{H_{\sigma\setminus\sigma_i}}(K\cap(y+H_\sigma))}^\frac{1}{m-s}dy.
    \end{split}
    \end{equation}
    Now let us define $f_i:P_{H_\sigma^\bot}K\rightarrow[0,\infty)$ given by
    \[
    f_i(y):=\vol{P_{H_{\sigma\setminus\sigma_i}}(K\cap(y+H_\sigma))}^\frac{1}{|\sigma|-|\sigma_i|}.
    \]
    Writing $p_i:=\frac{m-s}{|\sigma|-|\sigma_i|}|\sigma|$ for every $i=1,\dots,m$, we have that
    \[
    \frac{1}{p_1}+\cdots+\frac{1}{p_m} = \frac{1}{(m-s)|\sigma|}\sum_{i=1}^m(|\sigma|-|\sigma_i|)=1,
    \]
    and thus by H\"older inequality we then get that
    \begin{equation}\label{(2)}
    \begin{split}
    \int_{P_{H_\sigma^\bot}K} \prod_{i=1}^m f_i(y)^{\frac{|\sigma|-|\sigma_i|}{m-s}}dy & \leq 
     \prod_{i=1}^m\left( \int_{P_{H_\sigma^\bot}} f_i(y)^{p_i\frac{|\sigma|-|\sigma_i|}{m-s}} \right)^\frac{1}{p_i} \\
    & = \prod_{i=1}^m\left( \int_{P_{H_\sigma^\bot}} f_i(y)^{|\sigma|} \right)^\frac{|\sigma|-|\sigma_i|}{|\sigma|(m-s)}.
    \end{split}
    \end{equation}
    Let us notice that due to \eqref{eq:BrunnMinkowski} $f_i$ is a concave function for every $i\in[m]$. Hence, by Berwald inequality \eqref{eq:Berwald} we get that
    \begin{equation}\label{(3)}
    \left( \frac{{n\choose n-|\sigma|}}{\vol{P_{H_\sigma^\bot}K}} \int_{P_{H_\sigma^\bot}}f_i(y)^{|\sigma|}dy \right)^\frac{1}{|\sigma|} \leq
    \left( \frac{{n-|\sigma_i|\choose n-|\sigma|}}{\vol{P_{H_\sigma^\bot}K}} \int_{P_{H_\sigma^\bot}}f_i(y)^{|\sigma|-|\sigma_i|}dy \right)^\frac{1}{|\sigma|-|\sigma_i|}.
    \end{equation}
    Mixing \eqref{(1)}, \eqref{(2)}, and \eqref{(3)} we conclude that
    \[
    \begin{split}
    \mathrm{vol}(K)^{m-s} & \leq \frac{1}{{n\choose n-|\sigma|}^{m-s}\vol{P_{H_\sigma^\bot}K}^s} \prod_{i=1}^m{n-|\sigma_i|\choose n-|\sigma|}\int_{P_{H_\sigma^\bot}K}f_i(y)^{|\sigma|-|\sigma_i|}dy \\
    & = \frac{1}{{n\choose n-|\sigma|}^{m-s}\vol{P_{H_\sigma^\bot}K}^s} \prod_{i=1}^m{n-|\sigma_i|\choose n-|\sigma|}\vol{P_{H_{\sigma_i}^\bot}K}.
    \end{split}
    \]
\end{proof}
\end{document}